\documentclass[11pt, a4paper]{article}%
\usepackage{amsmath}%
\usepackage{amsfonts}%
\usepackage{amssymb}%
\usepackage{graphicx}
\usepackage{color}

\newtheorem{theorem}{Theorem}

\newtheorem{definition}[theorem]{Definition}

\newtheorem{lemma}[theorem]{Lemma}
\newtheorem{notation}[theorem]{Notation}

\newtheorem{proposition}[theorem]{Proposition}
\newtheorem{remark}[theorem]{Remark}

\newenvironment{proof}[1][Proof]{\noindent\textbf{#1.} }{\
\rule{0.5em}{0.5em}}

\def\A{\mathbb{A}}

\def\N{\mathbb{N}}

\def\Q{\mathbb{Q}}

\def\L{\mathbb{L}}
\def\V{\mathbb{V}}

\def\square{\ \rule{0.5em}{0.5em}}

\def\cF{\mathcal{F}}
\def\cG{\mathcal{G}}

\def\fP{\mathfrak{P}}

\begin{document}

\title{Effective Differential  L\"uroth's Theorem\thanks{
Partially supported by the following Argentinian grants: ANPCyT PICT 2007/816, UBACYT 2002010010041801 (2011-2014) and  UBACYT 20020090100069 (2010-2012).}}

\author{Lisi D'Alfonso$^\natural$ \and Gabriela Jeronimo$^\sharp$ \and Pablo Solern\'o$^\sharp$\\[4mm]
{\small $\natural$ Departamento de Ciencias Exactas, Ciclo B\'asico Com\'un, Universidad de Buenos Aires,} \\{\small Ciudad Universitaria, 1428, Buenos Aires, Argentina}\\[2mm]
{\small $\sharp$ Departamento de Matem\'atica and IMAS, UBA-CONICET,} \\{\small Facultad de Ciencias Exactas y Naturales,Universidad de Buenos Aires,}\\ {\small Ciudad Universitaria, 1428, Buenos Aires, Argentina}\\[3mm]
{\small E-mail addresses: lisi@cbc.uba.ar, jeronimo@dm.uba.ar, psolerno@dm.uba.ar}
}
\maketitle

\begin{abstract}
This paper focuses on effectivity aspects of the L\"uroth's theorem
in differential fields. Let $\mathcal{F}$ be an ordinary
differential field of characteristic $0$ and $\mathcal{F}\langle u
\rangle$ be the field of differential rational functions generated
by a single indeterminate $u$. Let be given non constant rational
functions $v_1,\ldots,v_n\in \mathcal{F}\langle u\rangle$ generating
a differential subfield $\mathcal{G}\subseteq \mathcal{F}\langle
u\rangle$. The differential L\"uroth's theorem proved by Ritt in
1932 states that there exists $v\in \mathcal G$ such that $\mathcal{G}= \mathcal{F}\langle v\rangle$. Here we prove that the total order and degree  of a generator $v$ are bounded by $\min _j \textrm{ord} (v_j)$ and $(nd(e+1)+1)^{2e+1}$, respectively, where $e:=\max_j \textrm{ord} (v_j)$ and $d:=\max_j \deg (v_j)$. As a byproduct, our techniques enable us to compute a L\"uroth generator by dealing with a polynomial ideal in a polynomial ring in finitely many variables.
\end{abstract}

%
%
%

\section{Introduction}

In 1876, J. L\"uroth in \cite{Luroth} presented his famous result, currently known as the  L\"uroth's Theorem: if $ k \subset L\subset k(u)$ is an extension of fields, where  $k(u)$ is the field of rational functions in one variable $u$, then $L=k(v)$ for a suitable  $v \in L$ (see \cite[\S 10.2]{vdW} for a modern proof). In 1893 G. Castelnuovo solved the same problem for rational function fields in two variables over an algebraically closed ground field. For three variables, L\"uroth's problem has been solved negatively.

In 1932 J.F. Ritt \cite{Ritt1} addressed the differential version of
this result: {\it Let $\cF$ be an ordinary differential field of
characteristic $0$, $u$ an indeterminate over $\cF$ and $\cF\langle
u \rangle$ the smallest field containing $\cF$, $u$ and all its
derivatives. Then, if $\cG$ is a differential field such that $\cF
\subset \cG \subset \cF\langle u \rangle$, there is  an element
$v\in \cG$ such that $\cG = \cF\langle v \rangle$.} Such an element
will be called a L\"uroth generator of the extension $\cF \subset
\cG$.

In fact, Ritt considered the case of a differential field $\cF$ of meromorphic functions in an open set of the complex plane and $\cG$ a finitely generated extension of $\cF$. Later, E. Kolchin in \cite{Kolchin2} and \cite{Kolchin3} gave a new proof of this theorem  for any differential field of characteristic $0$ and without the hypothesis of finiteness on $\cG$. Contrary to the classical setting, the differential L\"uroth problem fails in the case of two variables (see \cite{Ollivier}). A possible weak generalization of L\"uroth's theorem to dimension greater than one is the conjecture in control theory which states that every system linearizable by dynamic feedback is linearizable by endogenous feedback, or in algebraic terms, that a subextension of a differentially flat extension is differentially flat (\cite{FLMR}, \cite[Section 4.2]{Fliess}).

\bigskip

The present paper deals with quantitative aspects of the
Differential L\"uroth's Theorem and the computation of a L\"uroth
generator $v$ of a finite differential field extension. More
precisely (see Propositions \ref{orderv} and
\ref{deggenerator}):

\begin{theorem}
 Let $\cF$ be an ordinary differential field of
characteristic $0$, $u$ differentially transcendental over $\cF$ and
$\cG:= \cF\langle P_1(u)/Q_1(u), \dots, P_n(u)/Q_n(u)\rangle$, where
$P_j, Q_j\in \cF\{ u \}$ are relatively prime differential
polynomials of order at most $e\ge 1$ (i.e. at least one derivative
of $u$ occurs in $P_j$ or $Q_j$ for some $j$) and total degree
bounded by $d$ such that $P_j/Q_j\notin \cF$ for every $1\le j \le
n$. Then, any L\"uroth generator $v$ of $\cF\subset
\cG$ can
 be written as the quotient of two relatively prime differential polynomials $P(u), Q(u)\in \cF\{u\}$ with order
 bounded by $\min\{{\rm ord}(P_j/Q_j); 1\le j \le n\}$
and  total degree bounded by   $\min\{(d+1)^{(e+1)n}, (nd(e+1)+1)^{2e+1}$.
\end{theorem}


Our approach combines elements of Ritt's and Kolchin's proofs
(mainly the introduction of the differential polynomial ideal
related to the graph of the rational map $u\mapsto (P_j/Q_j)_{1\le
j\le n}$) with estimations concerning the order and the
differentiation index of differential ideals developed in
\cite{Sadik}, \cite{DJMS} and \cite{DJS06}. These estimations allow
us to reduce the problem of computing a L\"uroth generator to a Gr\"obner basis computation in a polynomial ring in finitely many variables (see Remark \ref{computation}).


\bigskip

An algorithmic version of Ritt's proof of the differential L\"uroth's Theorem is given in \cite{GaoXu}. The authors propose a deterministic algorithm that relies on the computation of ascending chains by means of the Wu-Ritt's zero decomposition algorithm; however, no quantitative questions on the order or the degree of the L\"uroth generator are addressed. For effectiveness considerations of the classical not differential version, we refer the interested reader to \cite{Netto}, \cite{Sederberg}, \cite{AGR}, \cite{GRS}, \cite{GRS2}, \cite{Cheze}.

\bigskip

This paper is organized as follows. In Section
\ref{sec:preliminaries} we introduce the notations,  definitions and
previous results from differential algebra (mainly concerning the
order and the differentiation index) needed in the rest of the
paper. In Section \ref{sec:problem} we present a straightforward
optimal upper bound on the order of any L\"uroth generator and we
discuss some ingredients that appear in the classical proofs of
L\"uroth's Theorem by Ritt and Kolchin and that we will use in our
arguments. In Section \ref{sec:reduction}, by means of estimates on
the differentiation index and the order of an associated DAE system,
we reduce the computation of a L\"uroth generator to an elimination
problem in effective classical algebraic geometry; as a byproduct,
we obtain upper bounds for the degree of a L\"uroth generator.
 Finally, in Section \ref{sec:examples} we show two simple examples illustrating our constructions.

\section{Preliminaries}\label{sec:preliminaries}

In this section we introduce the notation used throughout the paper and recall some definitions and results from differential algebra.

\subsection{Basic definitions and notation} \label{basic}

A \emph{differential field} $(\cF, \Delta)$ is a field $\cF$ with a
set of derivations $\Delta =\{ \delta_i\}_{i\in I}$, $\delta_i: \cF
\to \cF$. In this paper, all differential fields are \emph{ordinary}
differential fields; that is to say, they are equipped with only
\emph{one} derivation $\delta$; for instance, $\cF=\mathbb{Q}$,
$\mathbb{R}$ or $\mathbb{C}$ with $\delta=0$, or $\cF=\mathbb{Q}(t)$
with the usual derivation $\delta(t)=1$. For this reason, we will
simply write \emph{differential field} (instead of ordinary
differential field).

Let $(\cF, \delta)$ be a differential field of characteristic $0$.

The ring of differential polynomials in $\alpha$ indeterminates $z:=z_1,\ldots ,z_\alpha$,
which is denoted by $\cF\{z_1,\ldots ,z_\alpha\}$ or simply
$\cF\{z\}$, is defined as the commutative
polynomial ring $ \cF[z_j^{(p)}, 1\le j \le \alpha, \ p\in \N_0]$
(in infinitely many indeterminates), extending the derivation of
$\cF$ by letting $\delta(z_j^{(i)}) = z_j^{(i+1)}$, that is,
$z_j^{(i)}$ stands for the $i$th derivative of $z_j$ (as
customarily, the first derivatives are also denoted by $\dot z_j$).
We write $z^{(p)}:=z^{(p)}_1,\ldots,z^{(p)}_\alpha$ and
$z^{[p]}:=z, z^{(1)}, \dots, z^{(p)}$ for every $p\in \N_0$.

The
fraction field of $\cF\{z\}$ is a differential field,
denoted by $\cF\langle z \rangle$, with the derivation obtained by extending the derivation $\delta$ to the quotients in the usual way. For $g\in \cF\{ z \}$, the
\emph{order of $g$ with respect to $z_j$} is  ${\rm ord}(g,z_j) :=
\max\{i \in \N_0 : z_j^{(i)} \hbox{ appears in } g\}$, and the
\emph{order of $g$} is ${\rm ord}(g) := \max\{{\rm ord}(g,z_j) :
1\le j \le \alpha\}$; this notion of order extends naturally to
$\cF\langle z \rangle$  by taking the maximum of the orders of the numerator and the denominator in a reduced representation of the rational fraction.

Given differential polynomials $H:= h_1,\dots, h_\beta \in \cF\{
z\}$, we write $[H]$ to denote the smallest \emph{differential}
ideal of $\cF\{z\}$ containing $H$ (i.e.~the smallest ideal
containing the polynomials $H$ and all their derivatives of
arbitrary order).  The minimum \emph{radical} differential ideal of
$\cF\{ z \}$ containing $H$ is denoted by $\{ H \}$. For every $i\in
\N$, we write $H^{(i)}:=h_1^{(i)},\ldots,h_\beta^{(i)}$ and
$H^{[i]}:=H,H^{(1)},\ldots ,H^{(i)}$.

A \emph{differential field extension} $\cG/\cF$ consists of two
differential fields $(\cF, \delta_{\cF})$ and $(\cG, \delta_{\cG})$
such that $\cF \subseteq  \cG$ and $\delta_{\cF}$ is the
restriction to $\cF$ of $\delta_{\cG}$. Given a subset
$\Sigma\subset \cG$, $\cF\langle \Sigma \rangle$ denotes the minimal
differential subfield of $\cG$ containing $\cF$ and $\Sigma$.

An
element $\xi \in \cG$ is said to be \emph{differentially transcendental}
over $\cF$ if the family of its derivatives $\{\xi^{(p)} : p\in
\N_0\}$ is algebraically independent over $\cF$; otherwise, it is
said to be \emph{differentially algebraic} over $\cF$. A
\emph{differential transcendence basis} of $\cG/\cF$ is a minimal
subset $\Sigma\subset \cG$ such that the differential field
extension $\cG/\cF\langle \Sigma \rangle$ is differentially
algebraic. All the differential transcendence bases of a
differential field extension have the same cardinality (see \cite[Ch.~II,
Sec.~9, Theorem 4]{Kolchin}), which is called its \emph{differential
transcendence degree}.

\subsection{Differential polynomials, ideals and manifolds}\label{sec:rank}

Here we recall some definitions and properties concerning differential polynomials and their solutions.

Let $g\in \cF\{ z \} =\cF\{ z_1,\dots, z_\alpha\}$. The \emph{class} of $g$ for the order $z_1<z_2<\dots <z_\alpha$ of the variables is defined to be the greatest $j$ such that $z_j^{(i)}$ appears in $g$ for some $i\ge 0$ if $g\notin \cF$, and $0$ if $g\in \cF$.
If $g$ is of class $j>0$ and of order $p$ in $z_j$, the \emph{separant} of $g$, which will be denoted by $S_g$,  is $\partial g/ \partial z_j^{(p)}$ and the \emph{initial} of $g$, denoted by $I_g$, is the coefficient of the highest power of $z_j^{(p)}$ in $g$.

Given $g_1$ and $g_2$ in $\cF\{ z \}$, $g_2$ is said to be of \emph{higher rank in $z_j$ than} $g_1$ if either $\textrm{ord}(g_2, z_j) >\textrm{ord}(g_1, z_j)$ or $\textrm{ord}(g_2, z_j) =\textrm{ord}(g_1, z_j)= p$ and  the degree of $g_2$ in $z_j^{(p)}$ is greater than the degree of $g_1$ in $z_j^{(p)}$. Finally,  $g_2$ is said to be of \emph{higher rank than} $g_1$ if $g_2$ is of higher class than $g_1$ or they are of the same class $j>0$ and $g_2$ is of higher rank in $z_j$ than $g_1$.

We will use some elementary facts of the well-known theory of \textit{characteristic sets}. For the definitions and basic properties of rankings and characteristic sets, we refer the reader to \cite[Ch.~I, \S 8-10]{Kolchin}.

Let $H$ be a (not necessarily finite) system of differential polynomials in $\cF\{ z\}$. The \emph{manifold of $H$} is the set of all the zeros $\eta\in \cG^\alpha$ of $H$ for all possible differential extensions $\cG/\cF$.

Every radical differential ideal $\{ H\}$ of $\cF\{z\}$  has a unique representation as a finite irredundant intersection of prime differential ideals, which are called the \emph{essential prime divisors} of $\{ H\}$ (see \cite[Ch. II, \S 16-17]{Ritt}).

For a differential polynomial $g$ in $\cF\{z\}$  of positive class and algebraically irreducible, there is only one essential prime divisor of $\{ g \}$ which does not contain $S_g$; the manifold of this prime differential ideal is called the \emph{general solution of $g$} (see \cite[Ch. II, \S 12-16]{Ritt}).

\subsection{Hilbert-Kolchin function and differentiation index}

Let $\fP$ be a prime differential ideal of $\cF\{ z \}$. The
\emph{differential dimension} of  $\fP$, denoted by
$\textrm{diffdim}(\fP)$, is the differential transcendence degree of
the extension $\cF \hookrightarrow \textrm{Frac}(\cF\{z\}/\fP)$ (where Frac denotes the fraction field). The
\emph{differential Hilbert-Kolchin function} of $\fP$ with respect
to $\cF$ is the function $H_{\fP, \cF}: \N_0\to \N_0$ defined as:
\[H_{\fP, \cF}(i) :=\textrm{\ the (algebraic) transcendence degree of }\textrm{Frac}(\cF[z^{[i]}]/ (\fP \cap
\cF[z^{[i]}]))\ \textrm{over\ } \cF.\]

For $i\gg 0$, this function equals the linear function $
\textrm{diffdim}(\fP) (i+1) + \textrm{ord}(\fP),$ where
$\textrm{ord}(\fP)\in \N_0$ is an invariant called the \emph{order} of  $\fP$
(\cite[Ch.~II, Sec.~12, Theorem 6]{Kolchin}). The minimum $i$ from
which this equality holds is the Hilbert-Kolchin
\emph{regularity} of $\fP$.

Let $F$ be a finite set of differential polynomials contained in
$\fP$ of order bounded by a non-negative integer $e$. Throughout the
paper we assume that $e\ge 1$, in other words, all the systems we
consider are actually differential but not purely algebraic.

\begin{definition} \label{qr}
The set $F$ is \emph{quasi-regular} at $\fP$ if, for every $k\in \N_0$, the Jacobian matrix of
the polynomials $F, \dot F, \dots, F^{(k)}$ with respect to the variables $z^{[e+k]}$ has full row rank
over the fraction field of $\cF\{z\}/\fP$.
\end{definition}

A fundamental invariant associated to ordinary differential
algebraic equation systems is the \emph{differentiation index}.
There are several definitions of this notion  (see \cite{DJMS} and
the references given there), but in every case it represents a
measure of the implicitness of the given system. Here we will use
the following definition, introduced in \cite[Section 3]{DJMS}, in
the context of quasi-regular differential polynomial systems  with
respect to a fixed prime differential ideal $\fP$:

\begin{definition} \label{index} The \emph{$\fP-$differentiation index} $\sigma$ of a quasi-regular system $F$ of
polynomials in $\cF\{z\}$ of order at most $e$ is
\[\sigma := \min \{ k\in \N_0: (F, \dot F, \dots, F^{(k)})_{\fP_{e+k}} \cap \cF[z^{[e]}]_{\fP_e} = [F]_{\fP} \cap \cF[z^{[e]}]_{\fP_e} \},\]
where, for every $k\in \N_0$, $\fP_{e+k}:=\fP\cap \cF[z^{[e+k]}]$ (i.e. the contraction of the prime ideal $\fP$), $\cF[z^{[e+k]}]_{\fP_{e+k}}$ denotes the localized ring at the prime ideal $\fP_{e+k}$ and $(F, \dot F, \dots, F^{(k)})_{\fP_{e+k}}$ is the algebraic ideal generated by $F,\dot F, \dots, F^{(k)}$ in $ \cF[z^{[e+k]}]_{\fP_{e+k}}$.
\end{definition}

Roughly speaking, the differentiation index of the system $F$ is the
minimum number of derivatives of the polynomials $F$ needed to write
all the relations given by the differential ideal $\mathfrak{P}$ up
to order $e$.

\section{Differential L\"uroth's Theorem}\label{sec:problem}

In \cite[Chapter VIII]{Ritt1} (see also \cite{Ritt} and
\cite{Kolchin}), the classical L\"uroth's Theorem for transcendental
field extensions is generalized to the differential algebra
framework:

\begin{theorem}[Differential L\"uroth's Theorem] \label{luroth}
Let $\cF$ be an ordinary differential field of characteristic $0$
and let $u$ be differentially transcendental over $\cF$. Let $\cG$
be a differential field such that $\cF \subset \cG \subset
\cF\langle u \rangle$. Then, there is an element $v\in \cG$ such
that $\cG = \cF\langle v \rangle$.
\end{theorem}

Our goal is the following: for $n>1$, let be given differential
polynomials $P_1,\dots, P_n,$ $ Q_1,\dots, Q_n\in \cF\{ u \}$,
with $P_j/Q_j\notin \cF$ and $P_j, Q_j$ relatively prime polynomials
for every $1\le j\le n$, and denote \[\cG:= \cF \langle
P_1(u)/Q_1(u),\dots, P_n(u)/Q_n(u) \rangle,\] which is a subfield of
$\cF\langle u\rangle$. We want to compute a L\"uroth generator of
$\cG/\cF$, that is, a pair of differential polynomials $P, Q\in
\cF\{ u \}$ such that $Q \not\equiv 0$ and $\cG = \cF \langle P(u) /
Q(u)\rangle$.  We are also interested in the study of \textit{a
priori} upper bounds for the orders and degrees of both polynomials
$P$ and $Q$.

\bigskip
An optimal estimate for the order of the polynomials $P$ and $Q$ can
be obtained by elementary computations (see Section
\ref{subsec:order}). However, the problem of estimating their
degrees seems to be a more delicate question which requires a more
careful analysis that we will do in the subsequent sections of the
paper.

\subsection{Bound for the order}\label{subsec:order}

We start by proving an upper bound for the order of a L\"uroth generator.

\begin{proposition}\label{orderv}
Under the previous assumptions and notation,  any element $v\in
\cG$ such that $\cG=\cF\langle v \rangle$ satisfies ${\rm ord} (v)
\le  \min\{ {\rm ord}(P_j/Q_j): 1\le j \le n\}$.
\end{proposition}

\begin{proof}
Let $v\in \cG$ be such that $\cG = \cF\langle v \rangle$ and
$\varepsilon:={\rm ord}(v)$.

For $j=1,\dots, n$, let $v_j = P_j(u)/Q_j(u)$. By assumption, $v_j\notin \cF$. 
Let $T$ be a new differential indeterminate
over $\cF$. Since $v_j\in \cG=\cF\langle v \rangle$, there exists $\Theta_j\in \cF\langle
T\rangle$ such that $v_j = \Theta_j(v)$. Let $N_j= {\rm
ord}(\Theta_j)$. Then, $\textrm{ord}(v_j) \le N_j + \varepsilon$. In addition,
$$ \frac{\partial v_j}{\partial u^{(N_j+\varepsilon)}}= \frac{\partial (\Theta_j(v))}{\partial
u^{(N_j+\varepsilon)}} = \sum_{i\ge 0} \frac{\partial
\Theta_j}{\partial T^{(i)}}(v) \frac{\partial v^{(i)}}{\partial
u^{(N_j+\varepsilon)}} =\frac{\partial \Theta_j}{\partial
T^{(N_j)}}(v) \frac{\partial v^{(N_j)}}{\partial
u^{(N_j+\varepsilon)}}.
$$
Since $N_j$ is the order of $\Theta_j$, it follows that
$\frac{\partial \Theta_j}{\partial T^{(N_j)}}\ne 0$, and as $v$ is
differentially transcendental over $\cF$, we have that
$\frac{\partial \Theta_j}{\partial T^{(N_j)}}(v) \ne 0$. Furthermore, $\frac{\partial v^{(N_j)}}{\partial
u^{(N_j+\varepsilon)}} = \frac{\partial v}{\partial
u^{(\varepsilon)}}\ne 0$, since $\varepsilon= \textrm{ord}(v)$.  We conclude
that $\frac{\partial v_j}{\partial u^{(N_j+\varepsilon)}}\ne
0$ and, therefore, $\textrm{ord}(v_j) = N_j +\varepsilon$.

The proposition follows.
 \end{proof}

Note that the above proposition shows that all possible L\"uroth
generators $v$ have the same order. In fact, two arbitrary
generators are related by an homographic map with coefficients in
$\cF$ (see for instance \cite[\S 1]{Kolchin2}, \cite[Ch. II, \S
44]{Ritt}).

\subsection{Ritt's approach}

Here we discuss some ingredients which appear in the classical proofs of Theorem \ref{luroth} (see \cite{Ritt,Kolchin2}) and that we also consider in our approach.

Following \cite[II. \S39 and \S40]{Ritt}, let $y$ be a new
differential indeterminate over the field $\cF\langle u \rangle$
(and, in particular, over $\cG$) and consider the differential prime
ideal $\Sigma$ of all differential polynomials in $\cG\{y\}$
vanishing at $u$:
\begin{equation}\label{Sigma}
\Sigma :=\{A\in \cG\{ y\} \textrm{\ such\ that\ } A(u) = 0\}.
\end{equation}

\begin{lemma}\label{generalsolution}
The manifold of $\Sigma$ is the general solution  of an irreducible
differential polynomial $B \in \cG\{ y \}$. More precisely, $B$ is a
differential polynomial in $\Sigma$ with the lowest rank in $y$.
\end{lemma}

\begin{proof}
Let $B\in \Sigma$ be a differential polynomial with the lowest rank
in $y$. Note that $B\in \cG\{ y \}$ is algebraically irreducible,
since $\Sigma$ is prime. We denote the order of $B$ by $k$ and the
separant of $B$ by $S_B = \partial B/ \partial y^{(k)}$. Consider
the differential ideal
$$\Sigma_1(B): = \{ A \in \cG\{ y \} \mid S_B A \equiv 0 \mod \{ B\}
\}.$$ As shown in \cite[II. \S12]{Ritt}, the ideal $\Sigma_1(B)$ is
prime; moreover, we have that $A\in \Sigma_1(B)$ if and only if
$S_B^a A \equiv 0 \mod [B]$ for some $a\in \N_0$ and, in particular,
if $A\in \Sigma_1(B)$ is of order at most $k= {\rm ord}(B)$, then
$A$ is a multiple of $B$ (see \cite[II. \S13]{Ritt}). Furthermore,
$\Sigma_1(B)$ is an essential prime divisor of $\{ B\}$ and, in the
representation of $\{ B\}$ as an intersection of its essential prime
divisors, it is the only prime which does not contain $S_B$
(\cite[II. \S15]{Ritt}). Therefore, the manifold of $\Sigma_1(B)$ is the general solution of $B$.

In order to prove the lemma, it suffices to show that $\Sigma = \Sigma_1(B)$.

Let $A\in \Sigma_1(B)$. Then,  $S_B A \in \{B\}$. Taking into account that $\Sigma$ is
prime, $\{B\}\subset \Sigma$, and $S_B\notin \Sigma$, it follows that $A\in \Sigma$.

To see the other inclusion, consider a differential polynomial $A\in
\Sigma$. By the minimality of $B$, we have that $A$ is of rank at
least the rank of $B$. Reducing $A$ modulo $B$, we obtain a relation
of the type $$S_B^b I_B^c A \equiv R \mod [B],$$ where $b, c \in
\N_0$, $I_B$ is the initial of $B$ and $R$ is a differential
polynomial whose rank is lower than the rank of $B$. Since $A$ and
$B$ lie in the differential ideal $\Sigma$, it follows that $R\in
\Sigma$, and so, the minimality of $B$ implies that $R=0$. In
particular, $S_B^b I_B^c A \in [B]$ and, therefore, $I_B^c A\in
\Sigma_1(B)$. Now, $I_B\notin \Sigma_1(B)$ since, otherwise, it
would be a differential polynomial in $\Sigma$ with a rank lower
than the rank of $B$ (recall that $\Sigma_1(B)\subset \Sigma$); it
follows that $A\in \Sigma_1(B)$.
\end{proof}

\smallskip

Multiplying the polynomial $B\in \cG\{y\}$ given by Lemma \ref{generalsolution} by a suitable denominator, we obtain a differential polynomial $C\in \cF\{ u , y \}$ with no factor in $\cF\{ u \}$.  The following result is proved in \cite[II. \S42 and \S43]{Ritt}:
\begin{proposition} \label{ratio}
If $P_0(u)$ and $Q_0(u)$ are two non-zero coefficients of $C$
(regarded as a polynomial in $\cF\{u\}\{ y \}$) such that
$P_0(u)/Q_0(u) \notin \cF$, the polynomial \[D(u,y):= Q_0(u) P_0(y)
- P_0(u) Q_0(y)\] is a multiple of $C$ by a factor in $\cF$, and
$\cG= \cF\langle P_0(u)/Q_0(u) \rangle$. $\square$
\end{proposition}

Note that, by the definition of $C$, the ratio between two coefficients of $C$ coincides with the ratio of the corresponding coefficients of $B$.

\subsection{An alternative characterization of a L\"uroth
generator}\label{sec:characterization}

Under the previous assumptions, consider the map of differential algebras defined by
$$\begin{array}{ccc} \psi: \cF\{x_1,\dots, x_n,u\} &\to& \cF\{P_1(u)/Q_1(u),\dots, P_n(u)/Q_n(u),u\}\\
x_i &\mapsto & P_i(u)/Q_i(u)\\
u & \mapsto & u\end{array}
$$
Let $\fP\subset \cF\{x,u\}$ be the kernel of the morphism $\psi$; then, we have an isomorphism
\[\cF\{P_1(u)/Q_1(u),\dots, P_n(u)/Q_n(u),u\} \simeq \cF\{x_1,\dots, x_n,u\}/\fP.\]
This implies that $\fP$ is a \emph{prime} differential ideal and,
moreover, that the fraction field of $\cF\{x_1,\dots, x_n,u\}/\fP$
is isomorphic to $\cF\langle u \rangle$. In addition, the previous
isomorphism gives an inclusion
$$\cF\{P_1(u)/Q_1(u),\dots, P_n(u)/Q_n(u)\}\hookrightarrow \cF\{x_1,\dots, x_n,u\}/\fP,$$
and the inclusion induced from this map in the fraction fields leads
to the original extension $\cG= \cF\langle P_1(u)/Q_1(u),\dots,
P_n(u)/Q_n(u)\rangle \hookrightarrow \cF\langle u \rangle$. 

As above, let $y$ be a new differential indeterminate over $\cF
\langle u \rangle$ and $\Sigma$ the ideal of $\cG\{ y \}$ introduced
in (\ref{Sigma}). If
 $A\in \cG\{ y \}$ is a non-zero differential polynomial in $\Sigma$,
 multiplying it by an adequate element in $\cF\{P_1(u)/Q_1(u),\dots, P_n(u)/Q_n(u)\}$,
 we obtain a differential polynomial  in  $\cF\{P_1(u)/Q_1(u),\dots, P_n(u)/Q_n(u)\}\{ y \}$,
 with the same rank in $y$ as $A$. Taking a representative (with respect to $\psi$) in $\cF\{x_1,\dots, x_n\}$ for each of its coefficients, we get a differential polynomial  $\widehat A\in \cF\{x_1,\dots, x_n, y\}$, with the same rank in $y$ as $A$,  such that $\widehat A(x_1,\dots, x_n,u)\in
 \fP$.

Conversely, given a differential polynomial $M \in \cF\{x_1,\dots,
x_n, u\}$  such that $M\in \fP$ and not every coefficient of $M$ as a polynomial in $\cF\{x_1,\dots,x_n\}\{u\}$ lies in $\fP \cap \cF\{x_1,\dots, x_n\}$, the
 differential polynomial
 \begin{equation}\label{Mtilde}
 \widetilde M (y) := M(P_1(u)/Q_1(u), \dots, P_n(u)/Q_n(u), y)\in \cG\{ y \}
 \end{equation}
 is not the zero polynomial, vanishes at $u$ and has a rank in $y$ no higher than that of $M$.

We conclude that if $M\in \cF\{x_1,\dots, x_n, u\}$ is a
differential polynomial with the lowest rank in $u$ among all the differential polynomials as above, the
associated differential polynomial $\widetilde M(y)$ is a multiple
by a factor in $\cG$ of the minimal polynomial $B$ of $u$ over $\cG$
introduced in Lemma \ref{generalsolution}. Therefore, by Proposition
\ref{ratio}, a L\"uroth generator of $\cG/\cF$ can be obtained as
the ratio of any pair of coefficients of $\widetilde M\in \cG\{ y
\}$ provided that this ratio does not lie in $\cF$. Moreover:

\begin{proposition} \label{M1M2} Let $M\in \cF\{x_1,\dots, x_n, u\}$ be a
differential polynomial in $\fP \setminus (\fP \cap \cF\{x_1,\dots,x_n\})\{u\}$  with the lowest rank in $u$ and let
$\widetilde M (y) \in \cG\{ y \}$ be as in (\ref{Mtilde}). Assume
that $\widetilde M \in \cF(u^{[\epsilon]})\{ y  \}$ for a suitable
non-negative integer $\epsilon$. Consider two generic points
$\upsilon_1, \upsilon_2\in \Q^{\epsilon +1}$. Let $P(y)$ and $Q(y)$
be the differential polynomials obtained from $\widetilde M(y)$ by
substituting  $u^{[\epsilon]}=\upsilon_1$ and
$u^{[\epsilon]}=\upsilon_2$ respectively. Then $P(u)/Q(u)$ is a
L\"uroth generator of $\cG/\cF$.
\end{proposition}

\begin{proof} By Proposition \ref{ratio}, we have that
$$\widetilde M(y) = M(P_1(u)/Q_1(u), \dots, P_n(u)/Q_n(u), y) = \gamma (Q_0(u) P_0(y) - P_0(u) Q_0(y))$$
for some $\gamma \in \cG$, and $P_0, Q_0$ are such that $\cG =
\cF\langle P_0(u)/Q_0(u)\rangle$. Then, by means of two
specializations $\upsilon_1, \upsilon_2$ of the variables
$u^{[\epsilon]}$ so that $Q_0,Q_1,\dots,Q_n$ and $\gamma$ do not
vanish and $P_0(\upsilon_1)/Q_0(\upsilon_1) \ne P_0(\upsilon_2)
/Q_0(\upsilon_2)$, we obtain polynomials of the form
$$P(y) =\alpha_1 P_0(y) - \beta_1 Q_0(y)\qquad \hbox{and} \qquad
Q(y) =\alpha_2 P_0(y) - \beta_2 Q_0(y),$$ where $\alpha_1, \alpha_2,
\beta_1, \beta_2 \in \cF$ and $\alpha_1\beta_2 - \alpha_2\beta_1\ne
0$. The proposition follows since $\cF\langle P(u)/Q(u)\rangle=\cF
\langle P_0(u)/Q_0(u)\rangle=\cG$.
\end{proof}

\section{Reduction to a polynomial ring and degree bound}\label{sec:reduction}

In this section, we will obtain upper bounds for the order and the
degree of the differential polynomial $M\in \cF\{x_1,\dots, x_n,
u\}$ (see Proposition \ref{M1M2})  involved in our
characterization of a L\"uroth generator of $\cG/\cF$. These
bounds imply, in particular, an upper bound for the degrees of the
numerator and the denominator of the generator (see Section
\ref{degbou}).

\subsection{Bounding the order of a minimal polynomial $M$}

Here we estimate the order in the variables $x=x_1,\dots, x_n$ and
$u$ of a differential polynomial $M(x,u)\in \fP \setminus (\fP \cap \cF\{x_1,\dots,x_n\})\{u\}$ of minimal rank in
$u$, where $\fP$ is the prime differential ideal introduced in
Section \ref{sec:characterization}.

\begin{remark}\label{diffdimP} The differential dimension of $\fP$ equals $1$,  since the fraction
field of  $\cF\{x_1,\dots, x_n,u\}/\fP$ is isomorphic to $\cF\langle
u \rangle$.
\end{remark}

Let $e:= \max\{ {\rm ord}(P_j/Q_j): 1\le j \le n\}$. Without loss of generality, we may assume that ${\rm ord}(P_1(u)/Q_1(u)) \ge \cdots\ge {\rm ord}(P_n(u)/Q_n(u))$.
Consider the elimination order in $\cF\{x_1,\dots, x_n,u\}$ with $x_1<\cdots <x_n<u$.

Since  $P_1(u)/Q_1(u)\notin \cF$, it is transcendental over $\cF$.
Now, as the variable $u^{(e+1)}$ appears in the derivative
$(P_1(u)/Q_1(u))'$ but it does not appear in $P_1(u)/Q_1(u)$, it
follows that $(P_1(u)/Q_1(u))'$ is (algebraically) transcendental
over $\cF(P_1(u)/Q_1(u))$. Continuing in the same way with the
successive  derivatives, we conclude that $P_1(u)/Q_1(u)$ is
differentially transcendental over $\cF$. This implies that the
differential ideal $\fP$ contains no differential polynomial
involving only the variable $x_1$.

Thus, a characteristic set of $\fP$ for the considered elimination order is of the form $$R_1(x_1,x_2), R_2(x_1,x_2,x_3),\dots, R_{n-1}(x_1,\dots, x_n), R_n(x_1,\dots, x_n, u).$$
Furthermore, $R_n(x_1,\dots, x_n, u)$ is a differential polynomial in $\fP \setminus (\fP \cap \cF\{x_1,\dots,x_n\})\{u\}$ with a minimal rank in $u$, that is, we can take $M(x,u) = R_n(x,u)$. Following \cite[Lemma 19]{Sadik}, we may assume this characteristic set to be irreducible. Then, by \cite[Theorem 24]{Sadik}, we have that ${\rm ord} (R_i)\le {\rm ord}(\fP)$ for every $1\le i \le n$; in particular,
\begin{equation}\label{cotagrado1}
{\rm ord}(M)\le {\rm ord}(\fP).\end{equation}

The order of the differential prime ideal $\mathfrak P$ can be computed exactly. In order to do this, we introduce a system of differential polynomials that provides us with an alternative characterization of the ideal $\fP$ which enables us to compute its order.

 For every $j$, $1\leq j\leq n$, denote $F_j:=Q_j(u)x_j-P_j(u)\in \cF\{x,u\}$, and let $F:=F_1,\ldots,F_n$. Then we have:

\begin{lemma}\label{lem:minprime}
The ideal $\fP$ is the (unique) minimal differential prime ideal of
$[F]$ which does not contain the product $Q_1\dots Q_n$. Moreover,
$\fP = [F]:(Q_1\dots Q_n)^{\infty}$.
\end{lemma}

\begin{proof}
{}From the definitions of $\fP$ and $F$ it is clear that $[F]\subset
\fP$ and $Q_1(u)\dots Q_n(u)\notin \fP$. Moreover, since $F$ is a
characteristic set for the order in $\cF\{x,u\}$ given by
$u<x_1<\cdots <x_n$ and $\fP\cap \cF\{u\}=(0)$, we conclude that for
any polynomial $H\in \fP$ there exists $N\in \mathbb{N}_0$ such that
$(Q_1(u)\dots Q_n(u))^N H\in [F]$ (observe that $Q_j$ is the initial
and the separant of the polynomial $F_j$ for every $j$). The
proposition follows.
\end{proof}

The system $F$ we have introduced has the following property that we will use in the sequel (recall Definition \ref{qr}):

\begin{lemma} \label{lem:quasireg}
The system $F$ is quasi-regular at $\fP$.
\end{lemma}

\begin{proof}
Let $e$ be the maximum of the orders of the differential polynomials $F_j$, $1\le j \le n$.
For every $i\in \N$, let $J_i$ be the Jacobian matrix of the polynomials $F^{[i-1]}$ with respect to the variables $(x,u)^{[i-1+e]}$. We have that
$$
\frac{\partial F_j^{(k)}}{\partial x_h^{(l)}} =\begin{cases}
\ 0 & \textrm{if } h\ne j \textrm{ or } h=j,\  k<l\\[1mm]
\binom{k}{l} Q_j^{(k-l)} & \textrm{if } h=j,\ k\ge l
\end{cases}
$$
and so, for every $i\in \N$, the minor of $J_i$ corresponding to
partial derivatives with respect to the variables $x^{[i-1]}$ is a
scalar multiple of $(Q_1(u)\dots Q_n(u))^i$, which is not zero
modulo $\fP$.
 \end{proof}

Now, we apply results from \cite{DJMS} in order to compute the order of  $\fP$.

\begin{proposition}\label{orderP}
The order of the differential ideal $\fP$
equals $e= \max\{ {\rm ord} (P_j(u)/Q_j(u)): 1\le j \le n\}$.
\end{proposition}

\begin{proof}
Lemma \ref{lem:minprime} states that the ideal $\fP$ is an essential
prime divisor of $[F ]$, where $F:= F_1,\dots, F_n$ with $F_j(x,u)
:= Q_j(u) x_j - P_j(u)$, $1\le j \le n$, and, as shown in Lemma
\ref{lem:quasireg}, the system $F$ is quasi-regular at $\fP$.
Therefore, taking into account that $e$ is the maximum of the orders
of the polynomials in $F$, by \cite[Theorem 12]{DJMS}, the
regularity of the Hilbert-Kolchin function of  $\fP$ is at most
$e-1$. This implies that the order of $\fP$ can be obtained from the
value of this function at $e-1$; more precisely, since the
differential dimension of $\fP$ equals $1$, we have that
$$\textrm{ord}(\fP) = \textrm{trdeg}_\cF \Big(  \cF[x^{[e-1]}, u^{[e-1]}]/( \fP \cap \cF[x^{[e-1]}, u^{[e-1]}])\Big) - e.$$

In order to compute the transcendence degree involved in the above
formula, we observe first that \[\cF[x^{[e-1]}, u^{[e-1]}]/(\fP \cap
\cF[x^{[e-1]}, u^{[e-1]}] )\simeq \cF[ (P_j/Q_j)^{[e-1]},
u^{[e-1]}].\]

It is clear that the variables $u^{[e-1]}$ are algebraically
independent in this ring. Then, if $L= \cF(u^{[e-1]})$, the order of
the ideal $\fP$ coincides with the transcendence degree of $L(
(P_j/Q_j)^{[e-1]}, j=1,\dots, n)$ over $L$. Without loss of
generality, we may assume that $\textrm{ord}(P_1/Q_1) = e$. Since
the variable $u^{(e)}$ appears in $P_1/Q_1$, we have that
$L(u^{(e)})/L(P_1/Q_1)$ is algebraic. Similarly, since $u^{(e+1)}$
appears in $(P_1/Q_1)'$,
it
follows that the extension $L(u^{(e)},u^{(e+1)})/L((P_1/Q_1),(P_1/Q_1)')$ is
algebraic. Proceeding in the same way with the successive
derivatives of $ P_1/Q_1$, we conclude that $L(u^{(e)},u^{(e+1)},
\dots,u^{(2e-1)} )$ is algebraic over $L((P_1/Q_1)^{[e-1]})$.

Since $\textrm{ord}(P_j/Q_j) \le e$ for $j=1, \dots, n$, we have that $L( (P_j/Q_j)^{[e-1]}, j=1,\dots, n) \subset L(u^{(e)},u^{(e+1)}, \dots,u^{(2e-1)} )$ and, by the arguments in the previous paragraph, this extension is algebraic. Therefore,  $$\textrm{trdeg}_L L( (P_j/Q_j)^{[e-1]}, j=1,\dots, n)=\textrm{trdeg}_L L(u^{(e)},u^{(e+1)}, \dots,u^{(2e-1)} ) = e.$$
 \end{proof}

Then, by inequality (\ref{cotagrado1}) we conclude:

\begin{proposition}\label{orderM}
There is a differential polynomial $M\in \fP \setminus (\fP \cap \cF\{x_1,\dots,x_n\})\{u\}$ with the lowest rank in $u$ such that ${\rm ord}(M) \le e$. $\square$
\end{proposition}

\subsection{Reduction to algebraic polynomial
ideals}\label{sec:redalg}

As stated in Proposition \ref{M1M2}, the
L\"uroth generator of $\cG/\cF$ is closely related with
a polynomial $M(x_1,\dots, x_n, u)\in \fP \setminus (\fP \cap \cF\{x_1,\dots,x_n\})\{u\}$ with the
lowest rank in $u$.

By Proposition \ref{orderM},  such a polynomial $M$ can be found in the algebraic ideal $\fP \cap \cF[x^{[e]}, u^{[e]}]$ of the polynomial ring $ \cF[x^{[e]}, u^{[e]}]$. The following result will enable us to work with a finitely generated ideal given by known generators; the key point is the estimation of the $\mathfrak{P}$-differentiation index of the system $F:=F_1,\ldots,F_n$ (see Definition \ref{index}):

\begin{lemma}\label{diffindex}
The $\fP$-differentiation index of $F$ equals $e$. In particular,  we have that
$$[F]_{\fP} \cap \cF[x^{[e]}, u^{[e]}]_{\fP_e} = (F, \dot F, \dots, F^{(e)})_{\fP_{2e}} \cap
\cF[x^{[e]}, u^{[e]}]_{\fP_e},$$ where $\fP_e:=\fP \cap \cF[x^{[e]}, u^{[e]}]$ and $\fP_{2e}:= \fP \cap \cF[x^{[2e]}, u^{[2e]}]$.
\end{lemma}

\begin{proof}
For every  $k\in \N$, let $\mathfrak{J}_k$ be the Jacobian submatrix
of the polynomials $F, \dots, F^{(k-1)}$ with respect to the variables $(x,
u)^{(e)}, \dots, (x, u)^{(e+k-1)}$. The $\fP$-differentiation index of
$F$ can be obtained as the minimum $k$ such that $
\textrm{rank}(\mathfrak{J}_{k+1})- \textrm{rank}(\mathfrak{J}_k) =
n$ holds, where the ranks are computed over the fraction field of $\cF \{x, u\}/\mathfrak{P}$ (see \cite[Section 3.1]{DJMS}).

Now, since the order of the polynomials $F$ in the variables $x$ is
zero, no derivative $x^{(l)}$ with $l\ge e$ appears in $F, \dot F,
\dots, F^{(e-1)}$. This implies that the columns of the Jacobian
submatrices $\mathfrak{J}_k$ of these systems corresponding to
partial derivatives with respect to $x^{(l)}$, $l=e, \dots, e+k-1$,
are null. On the other hand, as $e$ is the order of the system $F$, we may suppose that
the variable $u^{(e)}$ appears in the polynomial $F_1$ and so,
$\partial F_1/\partial u^{(e)}\ne 0$. Thus, $\partial F_1^{(i)} / \partial u^{(h)}\ne 0$ for $h-i=e$ and $\partial F_j^{(i)} / \partial u^{(h)}=0$ for $h-i>e$; that is,  the matrices
$\mathfrak{J}_k$, $k=1,\dots, e$,  are block, lower triangular matrices of
the form
$$\mathfrak{J}_k = \left( \begin{array}{ccccccccccccccc}
0 & \cdots & 0 & * & \\
0 & \cdots & 0 & \star & 0 & \cdots & 0 & * & \\
\vdots & \vdots & \vdots &\vdots &\vdots &\vdots &\vdots &\vdots &\cdots & \\
0 & \cdots & 0 &  \star & 0 & \cdots & 0 & \star  & \cdots & 0 & \cdots & 0 & * \\
\end{array}
\right)$$
where $0$ denotes a zero column vector and $*$ a non-zero column vector.
Then, $\textrm{rank}(\mathfrak{J}_k) = k$ for $k=1,\dots, e$. Moreover,
$$\mathfrak{J}_{e+1} = \left( \begin{array}{cc}\mathfrak{J}_{e} & \\ \frac{\partial F^{(e)}}{\partial x^{(e)}} \ \cdots \ \cdots &
0\ \cdots\ 0\ * \end{array} \right)$$ and, by the diagonal structure
of $$\frac{\partial F^{(e)}}{\partial x^{(e)}} =
\begin{pmatrix}Q_1(u) & 0 & \cdots & 0  & \\0 & Q_2(u)& \ddots & \vdots \\\vdots  & \ddots& \ddots & 0 \\ 0 & \cdots& 0& Q_n(u)\end{pmatrix},$$ we see
that
$\textrm{rank}(\mathfrak{J}_{e+1})=\textrm{rank}(\mathfrak{J}_{e})+n$.

It follows that the $\fP$-differentiation index of the system equals $e$.
 \end{proof}

\begin{notation}\label{not:V}
We denote by $\V\subset \A^{(e+1) n}\times \A^{2e+1}$  the affine variety
defined as the Zariski closure of the solution set of the polynomial system
$$F=0, \dot F=0,\dots, F^{(e)}=0, \prod_{1\le j \le n} Q_j \ne 0,$$
where $F=F_1,\dots, F_n$ and $F_j(x, u^{[e]}) = Q_j(u^{[e]}) x_j - P_j(u^{[e]})$ for every $1\le j \le n$.
\end{notation}

The algebraic ideal corresponding to the variety $\V$ is $(F,\dot F,
\dots, F^{(e)}) : q^\infty$, where $q:= \prod_{1\le j \le n}  Q_j$.
This ideal is prime, since it is the kernel of the map
$$\begin{array}{ccc} \cF[x^{[e]}, u^{[2e]}] &\to& \cF[(P_j/Q_j)^{[e]}_{1\le j \le n}\ ,\ u^{[2e]}]\\[2mm]
x_j^{(k)} &\mapsto & (P_j/Q_j)^{(k)}\\
u^{(i)} & \mapsto & u^{(i)}\end{array}
$$
Then, $\V$ is an irreducible variety. Moreover, $F,\dot F,\dots,
F^{(e)}$ is a reduced complete intersection in $\{ q\ne 0\}$ and so,
the dimension of $\V$ is $2e+1$.

Now, the ideal of the variety $\V$ enables us to re-interpret the ideal
$\fP \cap \cF[x^{[e]}, u^{[e]}]$ where the minimal polynomial $M$
lies:

\begin{proposition} \label{ideals}  The following equality of ideals holds:
$$(F,\dot F,\dots, F^{(e)}) :q^{\infty} \cap \cF[x^{[e]}, u^{[e]}] = \fP\cap   \cF[x^{[e]}, u^{[e]}].$$
\end{proposition}

\begin{proof}
We start showing  that \[(F,\dot F,\dots, F^{(e)}) :q^{\infty} \cap \cF[x^{[e]}, u^{[e]}] = (F,\dot F,\dots, F^{(e)})_{\fP_{2e}} \cap   \cF[x^{[e]}, u^{[e]}]. \]

First note that $(F,\dot F,\dots, F^{(e)}) :q^{\infty}\subset  (F,\dot F,\dots, F^{(e)})_{\fP_{2e}}$ since $q\notin \fP$.  Conversely, if $h\in (F,\dot F,\dots, F^{(e)})_{\fP_{2e}} \cap   \cF[x^{[e]}, u^{[e]}]$, 
there is a polynomial $g\in \cF[x^{[2e]},u^{[2e]}]$ such that $g\notin \fP$ and $gh \in (F,\dot F,\dots, F^{(e)})$; but $g\notin (F,\dot F,\dots, F^{(e)}) :q^{\infty}$, since otherwise, $q^N g\in (F,\dot F,\dots, F^{(e)})\subset \fP$ for some $N\in \N$ contradicting the fact that $q\notin \fP$ and $g\notin \fP$. Since $(F,\dot F,\dots, F^{(e)}) :q^{\infty}$ is a prime ideal, it follows that $h \in  (F,\dot F,\dots, F^{(e)}) :q^{\infty}$.

Now, Lemma \ref{diffindex} implies that
\[(F,\dot F,\dots, F^{(e)})_{\fP_{2e}} \cap   \cF[x^{[e]}, u^{[e]}]= [F]_{\fP} \cap \cF[x^{[e]}, u^{[e]}].\]
 Finally, since $[F]_\fP \cF\{x,u\}_\fP = \fP_\fP \cF\{x,u\}_\fP$ (see \cite[Proposition 3]{DJMS}), we conclude that
\[
[F]_{\fP} \cap \cF[x^{[e]}, u^{[e]}] = \fP_\fP \cap \cF[x^{[e]}, u^{[e]}] =\fP\cap \cF[x^{[e]}, u^{[e]}];
\]
therefore,
\[(F,\dot F,\dots, F^{(e)}) :q^{\infty} \cap \cF[x^{[e]}, u^{[e]}] = \fP\cap \cF[x^{[e]}, u^{[e]}].\]
 \end{proof}

The previous proposition can be applied in order to effectively compute the polynomial $M(x_1,\dots, x_n,u)\in \fP \setminus (\fP \cap \cF\{x_1,\dots, x_n\})\{u\}$ with minimal rank in $u$, and consequently a L\"uroth generator of the extension $\cG/\cF$  (see Propositions \ref{ratio} and \ref{M1M2}), working over a polynomial ring in finitely many variables:

\begin{remark}\label{computation}
Consider the polynomial ideal $(F,\dot F,\dots, F^{(e)}): q^\infty \subset \cF[x^{[e]}, u^{[2e]}]$. Compute a Gr\"obner basis $G$ of this ideal for a pure lexicographic order with the variables $x^{[e]}$ smaller than the variables $u^{[2e]}$ and $ u<\dot u<\dots< u^{(2e)}$. Then, the polynomial $M$ is the smaller polynomial in $G$ which contains at least one variable in $u^{[2e]}$.
\end{remark}

\subsection{Degree bounds} \label{degbou}

In order to estimate the degree of a minimal polynomial
$M(x_1,\dots,x_n, u)\in \fP$ as in the previous section and, therefore, by Proposition
\ref{M1M2}, also the degree of a L\"uroth generator of $\cG/\cF$,
we
will relate $M$ to an eliminating polynomial for the algebraic
variety $\V$ under a suitable linear projection.

Let $k_0\in \N_0$ be the order of $u$ in $M$. By Proposition \ref{orderM}, we have that $k_0\le e$ and so, $M\in \cF[x^{[e]}, u^{[k_0]}]$.
Consider the fields $K:=\textrm{Frac}(\cF[x^{[e]}]/ (\fP \cap \cF[x^{[e]}]))$ and $L:= \textrm{Frac}(\cF[x^{[e]}, u^{[e]}]/ (\fP \cap \cF[x^{[e]},u^{[e]}]))$. The minimality of the rank in $u$ of $M$ is equivalent to the fact that $\{ u^{[k_0-1]}\}\subset L $ is algebraically independent over $K$ and $M$ is the minimal polynomial of $u^{(k_0)}\in L$ over $K(u^{[k_0-1]})$.

Let $Z\subset \{x^{[e]}\}$ be a transcendence basis of $K$ over $\cF$. Then, if we denote $U:=\{ u^{[k_0-1]}\}\subset L$, we have that $\{Z, U\}\subset L$ is  algebraically independent over $\cF$ and $\{Z, U, u^{(k_0)}\}\subset L$ is algebraically dependent over $\cF$ and, since $L\subset \cF(\V)$, the same holds in $\cF(\V)$.

Let $N$ be the cardinality of $\{Z, U\}$. Consider the projection
\begin{equation}\label{eq:pi}
\pi: \V \to \A^{N+1}, \quad \pi(x^{[e]},
u^{[2e]})= (Z,U, u^{(k_0)}).
\end{equation}
For the construction of $Z,U$, the dimension of $\pi(\V)$ equals $N$ and so, the Zariski closure of $\pi(\V)$
is a hypersurface in $\A^{N+1}$. Let $M_0\in \cF[Z,U,u^{(k_0)}]$ be an irreducible polynomial
defining this hypersurface (recall that $\V$ is an irreducible variety). From \cite[Lemma 2]{Heintz83}, we have the inequality
\begin{equation}\label{degMdegV}
\deg M_0 \le \deg \V
\end{equation}

Note that $M_0\in \fP \setminus (\fP \cap \cF\{x\})\{u\}$ and $\textrm{ord}_u (M_0) = k_0= \textrm{ord}_u (M)$. However, $M_0$ may not have the property of minimal rank in $u$, as the following simple example shows:

\bigskip

\noindent \textbf{Example 1.} Let $\cG= \cF\langle  \dot u, (\dot u)^2\rangle$. Here:
\begin{itemize}
\item $e=1$, $n=2$
\item $F_1 =x_1-\dot u$, $F_2 = x_2 - (\dot u)^2$, $q=1$
\end{itemize}
Following our previous construction,
$$\V= V(F_1, F_2, \dot F_1, \dot F_2) = V(x_1-\dot u, x_2 - (\dot u)^2, \dot x_1-\ddot u, \ddot x_2 - 2 \dot u \ddot u)$$
and we can take $Z= \{x_2, \dot x_2\}$ and $U=\{u\}$. Then, $M_0= (\dot u)^2 - x_2$, which has not minimal rank in $u$, since $\dot u -x_1$ vanishes over $\V$. In fact, $M= \dot u -x_1$.

\bigskip

Even though the polynomial $M_0$ is not the minimal rank polynomial $M$ we are looking for, the following relation between their degrees will be sufficient to obtain a degree upper bound for a L\"uroth generator.

\begin{proposition}
With the previous assumptions and notation,  we have that $\deg_{(U,u^{(k_0)})}(M)\le  \deg_{(U,u^{(k_0)})}(M_0)$. In particular, $\deg_{(U,u^{(k_0)})}(M)\le \deg(\V)$.
\end{proposition}

\begin{proof}
By construction, $M$ is the minimal polynomial of $u^{(k_0)}$ over $K(U)$ and $M_0$ is the minimal polynomial of $u^{(k_0)}$ over $\cF(Z,U)$. Without loss of generality, we may assume that $M$ and $M_0$ are polynomials with coefficients in $\cF$ having content $1$ in $K[U]$ and $\cF[Z,U]$ respectively. Since $\cF(Z, U) \subset K(U)$, we infer that $M$ divides $M_0$ in $K[U][u^{(k_0)}]$. The proposition follows.
\end{proof}

Recalling that a L\"uroth generator $v$ of $\cG/\cF$ can be obtained as the quotient of two specializations of the variables $x_1,\dots, x_n$ and their derivatives in the polynomial $M$ (see Proposition \ref{M1M2}) and that two arbitrary generators are related by an homographic map with coefficients in $\cF$ (see \cite[\S 1]{Kolchin2}, \cite[Ch. II, \S 44]{Ritt}), we
conclude that the degrees of the numerator and the denominator of any L\"uroth generator of $\cG/\cF$ are bounded by the degree of the variety $\V$.

\bigskip

We can exhibit purely syntactic degree bounds for $\V$ in terms of the number
$n$ of given generators for $\cG/\cF$, their maximum order $e$, and an upper bound
$d$ for the degrees of their numerators and denominators.

First, since the variety $\V$ is an irreducible component of the algebraic set defined by the $(e+1)n$ polynomials
$F, \dot F,\dots , F^{(e)}$ of total degrees bounded by $d+1$ (here $F= F_1,\dots, F_n$
and $F_j(x,u) = Q_j(u) x_j - P_j(u)$), B\'ezout's theorem (see for instance \cite[Theorem 1]{Heintz83}) implies
that
$$\deg \V \le (d+1)^{(e+1)n}.$$

An  analysis of the particular structure of the system leads to a
different upper bound for $\deg(\V)$, which is not exponential in $n$: taking into
account that $\V$ is an irreducible variety of dimension $2e+1$, its
degree is the number of points in its intersection  with a generic
linear variety of codimension $2e+1$, that is,
$$\deg \V = \# ( \V \cap V(L_1,\dots, L_{2e+1})),$$
where, for every $1\le i \le 2e+1$, $L_i$ is a generic affine linear
form in the variables $x^{[e]}, u^{[2e]}$,
\begin{equation}\label{eq:linearforms}
L_i(x^{[e]}, u^{[2e]}) = \sum_{1\le j \le n\atop 0\le k \le e}
a_{ijk}\, x_j^{(k)} + \sum_{0\le k \le 2e} b_{ik} \, u^{(k)} + c_i.
\end{equation}

For every $1\le j \le n$, the equation $F_j(x_j,u)=0$ implies that,
generic points of $\V$ satisfy $x_j = P_j(u)/Q_j(u)$. Proceeding
inductively, it follows easily that, generically
$$x_j^{(k)} = \left(\frac{P_j(u)}{Q_j(u)}\right)^{(k)}=
\frac{R_{jk}(u^{[e+j]})}{Q_j(u^{[e]})^{k+1}}\qquad \hbox{ with }
\deg(R_{jk}) \le d(k+1)$$ for every $1\le j \le n, \, 0\le k \le e$.
Substituting these formulae into (\ref{eq:linearforms}) and clearing
denominators, we deduce that the degree of $\V$ equals the number of
common solutions of the system defined by the $2e+1$ polynomials
\begin{eqnarray*}
\L_i(u^{[2e]}) &:=& \Big(\prod_{1\le j \le n} Q_j^{e+1}\Big)\
L_i\left(\Big(\frac{P_1}{Q_1}, \dots, \frac{P_n}{Q_n}\Big)^{[e]},
u^{[2e]}\right) \\
&=& \sum_{1\le j \le n\atop 0\le k \le e} a_{ijk}\, R_{jk}
Q_j^{e-k}\prod_{h\ne j} Q_h^{e+1} + \sum_{0\le k \le 2e} b_{ik} \,
u^{(k)} \prod_{1\le j \le n} Q_j^{e+1} + c_i \prod_{1\le j \le n}
Q_j^{e+1}
\end{eqnarray*} for generic coefficients $a_{ijk}, b_{ik}, c_i$ and the inequality $q \ne 0$.
{}From the upper bounds for the degrees of the polynomials $Q_j,
R_{jk}$, it follows that $\L_i$ is a polynomial of total degree
bounded by $nd(e+1)+1$ for every $1\le i \le 2e+1$. Therefore,
B\'ezout's bound implies that
$$\deg \V \le (nd(e+1)+1)^{2e+1}.$$

We conclude that:
%
%

\begin{proposition}\label{deggenerator}
Under the previous assumptions and notation, the degrees of the numerator and the denominator of any L\"uroth generator of $\cG/\cF$ are bounded by $\min\{ (d+1)^{(e+1)n}$, $(nd(e+1)+1)^{2e+1}\}$. $\square$
\end{proposition}

\section{Examples}\label{sec:examples}

As before, let $\cF$ be a differential field of
characteristic $0$ and  $u$ a differentially transcendental element
over $\cF$.

\bigskip

\noindent \textbf{Example 1.}  Let $\cG = \cF\langle u/\dot u,\, u + \dot u\rangle$.
In this case, we have:
\begin{itemize}
\item $e=1$, $n=2$,
\item $F_1= \dot u x_1 - u$, $F_2= x_2 - u - \dot u$, $q= \dot u$,
\end{itemize}
As the ideal $(F_1, F_2, \dot F_1, \dot F_2) = (\dot u  x_1 - u,\, x_2 - u - \dot
u,\, \ddot u x_1 + \dot u \dot x_1 - \dot u, \dot x_2 - \dot u
-\ddot u)$ is prime and does not contain $\dot u$, we have that $(F_1, F_2, \dot F_1, \dot F_2): q^\infty = (F_1, F_2, \dot F_1, \dot F_2)$ and, therefore,
$$\V = V(F_1, F_2, \dot F_1, \dot F_2) = V(\dot u  x_1 - u,\, x_2 - u - \dot
u,\, \ddot u x_1 + \dot u \dot x_1 - \dot u, \dot x_2 - \dot u
-\ddot u).$$

The dimension of $\V$ equals $3$ and $\{x_1,x_2,
\dot x_1\}$ is a transcendence basis of $\cF(\V)$  over $\cF$.
Then, $k_0 = 0$ and so, we look for a polynomial $M_0(x_1, x_2, \dot
x_1, u) \in (F_1, F_2, \dot F_1, \dot F_2)$. We compute this
polynomial by an elimination procedure:
$$M_0(x_1,x_2,\dot x_1, u) = (x_1+1)\, u - x_1 x_2.$$
Since $\deg_u(M_0) =1$, then $M=M_0$.
Finally, specializing $(u,\dot u , \ddot u)$ into $u_1= (1,1,0)$ and
$u_2 = (0,1,0)$, we compute specialization points $(1, 2,1)$ and
$(0, 1, 1)$ respectively for $(x_1, x_2, \dot x_1)$; hence, we
obtain the following L\"uroth generator for $\cG/\cF$:
$$v = \frac{P(u)}{Q(u)} = \frac{2u -2}{u}.$$

\bigskip

In the previous example, the polynomial $M$ lies in the polynomial ideal $(F_1, F_2)$, that is, no differentiation of the equations is needed in order to compute it and, consequently, a L\"uroth generator can be obtained as an algebraic rational function of the given generators. However, this is not always the case, as the following example shows.

\bigskip
\noindent \textbf{Example 2.} Let $\cG = \cF\langle \dot u, u+ \ddot u\rangle$.
We have:
\begin{itemize}
\item $e=2$, $n=2$,
\item $F_1= x_1 - \dot u$, $F_2= x_2 - u -  u^{(2)}$, $q= 1$,
\end{itemize}
Following our previous arguments, we consider the ideal \[(F_1, F_2, \dot F_1, \dot F_2, F_1^{(2)}, F_2^{(2)})=\] \[= (x_1 - \dot u, x_2 - u - u^{(2)}, \dot x_1 - u^{(2)}, \dot x_2 - \dot u - u^{(3)}, x_1^{(2)} - u^{(3)},  x_2^{(2)} -  u^{(2)} - u^{(4)}).\]
 This is a prime ideal of $\cF[x^{[2]}, u^{[4]}]$. The variety $\V$ is the zero-set of this ideal and it has dimension $5$. A transcendence basis of $\cF(\V)$ over $\cF$ including a maximal subset of $x^{[2]}$ is $\{x_1, x_2,\dot x_1,\dot x_2, x_2^{(2)}\}$ and the minimal polynomial of $u$ over $\cF(x_1, x_2,\dot x_1,\dot x_2, x_2^{(2)})$ is $$M_0=u - x_2 +\dot x_1.$$
Again, since $\deg_u(M_0) = 1$, then $M=M_0$.
Two specializations of this polynomial lead us to a L\"uroth generator of $\cG/\cF$ of the form $v= \dfrac{u+a}{u+b}$. We conclude that $\cG = \cF\langle u \rangle$.





\bibliographystyle{model1b-num-names}







\end{document}